\documentclass[reqno]{amsart}
\usepackage{mathrsfs}
\usepackage{amscd}
\usepackage[dvips]{graphics}

\def\beq{\begin{equation}}
\def\eeq{\end{equation}}
\def\ba{\begin{array}}
\def\ea{\end{array}}

\newtheorem{thm}{Theorem}[section]
\newtheorem{lem}[thm]{Lemma}
\newtheorem{prop}[thm]{Proposition}
\newtheorem{crl}[thm]{Corollary}

\theoremstyle{definition}
\newtheorem{rem}[thm]{Remark}

\theoremstyle{remark}
 
\numberwithin{equation}{section}

\begin{document}
\allowdisplaybreaks[3]
\pagestyle{plain}   
\title{Hardy-Littlewood-Sobolev inequalities on compact Riemannian manifolds and applications}

\author{Yazhou Han}
\address{Yazhou Han,
Department of Mathematics,College of Science,
China Jiliang University,
Hangzhou, 310018, China}
\email{yazhou.han@gmail.com}

\author  {Meijun Zhu}
\address{ Meijun Zhu, Department of Mathematics,
The University of Oklahoma, Norman, OK 73019, USA}
\email{mzhu@ou.edu}



\begin{abstract}
In this paper we extend  Hardy-Littlewood-Sobolev inequalities on  compact Riemannian manifolds for dimension $n\ne 2$.
As one application, we solve a generalized Yamabe problem on locally  conforamlly flat manifolds via a new designed energy functional and a new variational approach. Even  for the classic  Yamabe problem on locally conformally flat manifolds,   our  approach provides  a new and relatively simpler solution.


\end{abstract}

 \maketitle

\section{Introduction}

%

Curvature equations involving high order derivatives (including $Q-$ curvature equations) and  fully nonlinear curvature  equations (such as $\sigma_k$ operators of Schouten tensor) have been extensively studied in the past decade, and have broad applications in the study of global geometry and topology. See, e.g. \cite{CGY2002}, \cite{GV2003}, \cite{Br2003}, \cite{LL2005}, \cite{GLW2004}, \cite{GV2007}, \cite{DM2008}  and references therein.  All these differential operators, such as Paneitz operators with even powers and $\sigma_k$ operators of Schouten tensor, are introduced as a locally defined operators.

Recently, there have been some interesting results concerning the fractional Yamabe problem, as well as the fractional prescribing curvature problem, see, e.g. \cite{GZ2003},\cite{GM2011}, \cite{GQ2011}, \cite{JLX2011a}-\cite{JX2011} and references therein. In these studies  the notion for the globally defined fractional Paneitz operator $P_\alpha$ (via an integral operator), which is  introduced in \cite{GZ2003}, is used  and has a direct link to singular integral operators (see Caffarelli and Silvestre \cite{CS2007} for a new view point of fractional Laplacian operator).

 Motivated by the globally defined fractional Paneitz operator, as well as the study of sharp Sobolev inequality with negative power by W. Chen, et al \cite{CHLYZ}, Yang and Zhu \cite{YZ2004}, Hang and Yang \cite{HY2004}, Ni and Zhu \cite{NZ1}-\cite{NZ3}, Hang \cite{H2007}, etc. we started to investigate the general extension of Hardy-Littlewood-Sobolev (HLS) inequality. In Dou and Zhu \cite{DZ1}, we established the HLS inequality on the upper half space, and outline the rough idea on the extension of HLS on general manifolds; In Dou and Zhu \cite{DZ2}, a surprising reversed HLS inequality was obtained when the differential order is higher than the dimension. In Zhu \cite{Zhu2014}, a more general prescribing curvature equation on $\mathbb{S}^n $ was introduced and the existence result for antipodally symmetric function was obtained; in particular, the reversed HLS inequality was first used in the study of curvature equations with negative critical Sobolev exponents.  In the same paper, a more general Yamabe type problem was also  introduced for general compact Riemannian manifolds.  In this paper we shall extend the classic HLS inequality as well as the reversed HLS inequality on  compact Riemannian manifolds and provide solution to the general Yamabe problem on locally conformally flat manifolds.

\smallskip

Let $(M^n, g)$ be a given compact  Riemmanian manifold, $\alpha$($ \ne n$)  be a positive parameter and $|x-y|_g$  represent the distance from $x$ to $y$ on $M^n$ under metric $g$.
Introduce the following integral operator:
\begin{align*}
    I_\alpha f(x)=\int_{M^n}\frac{f(y)}{|x-y|_g^{n-\alpha}}dV_y,\label{Operator-main-term}\\
\end{align*}
We first have the following HLS inequality on $(M^n, g)$ for $\alpha<n$:

\begin{prop}\label{prop-sharp-ineq}
 Assume that $\alpha \in (0, n)$,  $1<p<\frac n\alpha$ and $q$ is given by \begin{equation}\label{power} \frac1q=\frac1p-\frac\alpha n,
 \end{equation} then  there is an optimal positive constant $C(\alpha,p,M^n,g)$, such that
\begin{equation}\label{1-1}
    ||I_{\alpha} f||_{L^q(M^n)} \le C(\alpha, p, M^n, g)||f||_{L^p(M^n)}.
\end{equation}
holds for all  $f\in L^p({M}^n).$
Moreover, for $1\le r<q$,  operator $I_{\alpha} : L^p(M^n)  \to  L^r(M^n)$ is a compact  embedding.

\end{prop}

Proposition \ref{prop-sharp-ineq} seems to be a known fact. Since we can not find the proof in literatures, we will outline the proof in this paper.


For $\alpha>n$, we have the following reversed HLS inequality for nonnegative functions.

\begin{thm}\label{prop-rough-ineq_R}
 Assume that $\alpha>n\geq 1$,  $1>p>\frac n\alpha$ and $q$ is given by \eqref{power}, then
 there is an optimal positive constant $C(\alpha,p,M^n,g)$, such that
\begin{equation}\label{1-1-1}
    ||I_{\alpha} f||_{L^q(M^n)} \ge C(\alpha, p, M^n, g)||f||_{L^p(M^n)}.
\end{equation}
holds for all nonnegative  $f\in L^p({M}^n).$


\end{thm}

One of the main motivations for obtaining  the above embedding theorems  comes from the  study of curvature equations, including the following  generalized Yamabe problem, introduced in Zhu \cite{Zhu2014}:

For a given compact Riemannian manifold $(M^n, g_0)$ ($n\ne 2$) with positive scalar curvature and a positive parameter $\alpha \ne n+2k$ for $k=0, 1, \cdots$, let $G^{g_0}_x(y)=n(n-2)\omega_n \Gamma_x^{{g_0}}(y)$, where $ \Gamma_x^{{g_0}}(y)$  is the Green's function with pole at $x$ for the  conformal Laplacian operator $-\Delta_{g_0} +\frac{n-2}{4(n-1)}R_{g_0}$, $\omega_n$ is the volume of the unit ball. In a conformal normal coordinates centered at $x$, $ G^{g_0}_x(y)=|y|^{2-n}+A+O(|y|)$.  The $\alpha-$ curvature $Q_{\alpha, g}$ under the conformal metric $g=\phi^{\frac 4{n-\alpha}}g_0$ is defined as a function implicitly given by
 \begin{equation}\label{curvature}
u(x)= \int_{\mathbb{M}^{n}}[G_{x}^{g_0}(y)]^{\frac {\alpha-n}{2-n}} {Q_{\alpha,{g}} (y)u^{\frac{n+\alpha}{n-\alpha}}(y)}dV_{g_0}.
\end{equation}
It is clear that $Q_{\alpha, g}$, up to a constant multiplier,  is the classic scalar curvature for $\alpha=2$.

Let
 \begin{equation}\label{new_operator}
   I_{M^n, g, \alpha} (f)=\int_{\mathbb{M}^{n}}[G^{g}_x(y)]^{\frac {\alpha-n}{2-n}} f(y) dV_{g}.
  \end{equation}
  It was showed in \cite{Zhu2014} that $I_{M^n, g, \alpha} (f)$ has the conformal covariance property. Similar to the Yamabe problem, one may ask \cite{Zhu2014}: for a given compact Riemannian manifold $(M^n, g_0)$, is there a conformal metric $g=u^{4/(n-\alpha)}g_0$ such that $Q_{\alpha, g}=constant$? We shall solve this problem on any locally conformally flat manifold with positive scalar curvture, based on the positive mass theorem.\footnote {We need more careful expansion for the Green's function $G^{g_0}_x(y)$ in a normal coordinate in order to work on locally conformally non-flat cases. We thank F. Hang who pointed out this sutble issue to us.}

  \begin{thm}
  For a given compact  locally conformally flat manifold  $(M^n, g)$ ($n\ne 2$) with positive scalar curvture, there always exists   a conformal metric $g_*=u^{4/(n-\alpha)}g$ such that $\alpha-$curvature $Q_{\alpha, g_*}$ is a constant.
  \label{yamabe}
  \end{thm}

From now on in this paper, we always assume the compact manifold $(M^n, g)$ under consideration has positive scarlar curvature.

The traditional approach to solve the classic Yamabe problem is to seek the minimizer to the Sobolev quotient energy:
$$J_2(u)=\frac{\int_{M^n}{|\nabla u|^2+\frac{n-2}{4(n-1)}R_{g_0}u^2} dV_{g_0}}{\|u\|^2_{L^{2n/(n-2)}(M^n)}}.
$$
Unfortunately, for fractional order $\alpha$, such an energy functional is hard to find.   To prove the above theorem, we  design the following energy   functional for positive functions:

\begin{equation}
    J_{g,\alpha}(u):=\frac{\int_{M^n}\int_{M^n}{u(x)u(y)}[G^{g}_x(y)]^{\frac {\alpha-n}{2-n}}dV_g(x)dV_g(y)}{\|u\|^2_{L^{2n/(n+\alpha)}(M^n)}}.
    \label{infimum}
\end{equation}
The above functional was successfully used in \cite{Zhu2014} to solve a prescribing curvature problem on $\Bbb{S}^n$ with negative  exponent (in the case of $\alpha>n$). In this paper, we will show that it can  also be used to solve  Yamabe type problems.

 For $\alpha<n$, we consider the supremum
\begin{equation}\label{min}
Y_\alpha (M^n, g):=\sup_{u\in C^0(M^n)\setminus \{o\}, u \ge 0} J_{g,\alpha}(u).
\end{equation}
Similar to the proof of Proposition \ref{prop-sharp-ineq}, one can show that $Y_\alpha (M^n, g)<\infty$ ({see remark \ref{rem2.2} below}). Moreover, it follows from  Lieb's classic result \cite{Lieb1983}, that the supremum on  the standard sphere $(\mathbb{S}^n, g_0)$ or on flat plane ($\mathbb{R}^n,$ $g_E$)  is given by
\begin{equation}\label{8-5-1}
Y_\alpha(\mathbb{S}^n, g_0) =\pi^{\frac{n-\alpha}2}\frac{\Gamma(\frac\alpha 2)}{\Gamma(\frac{n+\alpha}2)} \left\{\frac{\Gamma(\frac n2)}{\Gamma(n)}\right\}^{-\frac\alpha n}
\end{equation}
and the corresponding extremal functions on ($\mathbb{R}^n,$ $g_E$) are $f(x)=(1+|x|^2)^{-\frac{n+\alpha}2}$ and its conformal eqivalent class:
\begin{equation}\label{buble}
f_{\epsilon, x_0}(x)=\epsilon^{-\frac {n+\alpha}2}f(\frac {x-x_0}\epsilon)=\left(\frac\epsilon{\epsilon^2+|x-x_0|^2}\right)^{\frac{n+\alpha}2}
\end{equation}
where $x_0 \in \mathbb{R}^n$ and  $\epsilon>0$. For convenience, we  write $f_{\epsilon}=f_{\epsilon, 0}$ in this paper.

We will first show that $ Y_\alpha (M^n, g) \ge Y_\alpha (\mathbb{S}^n, g_0)$;  and for a   locally conformally flat manifold  $(M^n, g)$  with positive scalar curvture, the equality holds iff  $(M^n, g)$  is conformally equvalent to  $(\mathbb{S}^n, g_0)$.
As in the study of Yamabe type problem, we will  then show that
 the strict inequality yields the existence of the maximizer, based on a new $\epsilon-$level sharp HLS inequality on manifolds (Proposition \ref{epsilon_ineq_prop} below).  This approach will give a new view point even for the proof of the classic Yamabe problem. We recently learned from Hang and Yang  that such approach was also used  in their recent work \cite{HY2014-1} for $Q$-curvature problem ($\alpha =4$ in their case).

\smallskip

Parallel to the case of $\alpha<n$, for $\alpha>n$, we consider the infimum
\begin{equation}\label{max}
Y_\alpha (M^n, g):=\inf_{u\in C^0(M^n)\setminus \{o\}} J_{g,\alpha}(u).
\end{equation}
It follows from Theorem \ref{prop-rough-ineq_R} that $Y_\alpha (M^n, g)>0$. And, it follows from  the sharp reversed HLS inequality \cite{DZ2}  that the infimum on  the standard sphere or flat plane is given by \eqref{8-5-1} and
 the corresponding extremal functions on ($\mathbb{R}^n,$ $g_E$) are given by \eqref{buble}.  Again, we will first  show that $ Y_\alpha (M^n, g) \le Y_\alpha (\mathbb{S}^n, g_0)$, and for a   locally conformally flat manifold $(M^n, g)$  with positive scalar curvture, equality holds iff  $(M^n, g)$ is conformally equvalent to  $(\mathbb{S}^n, g_0)$.  We then  show that
 the strict inequality yields the existence of the minimizer through a new blowup analysis. It is interesting to point out that the local blowup analysis does not work due to the lack of local Sobolev inequality for $\alpha>n$.

 This paper is organized as follows. In Section \ref{case_le}, we deal with the case of $\alpha<n$.  Based on the Marcinkiewicz interpolation theorem, we prove  the roughly HLS inequality \eqref{1-1} on $(M^n, g)$ and the compactness of embedding for subcritical exponent. We then  establish an $\epsilon$-level sharp HLS inequality on any general compact manifold and complete the proof of Theorem \ref{yamabe} for $\alpha<n$. In Section \ref{case_ge}, we deal with the case of $\alpha>n$. The analog $\epsilon$-level inequality is not known. Instead, a new blow up analysis enables us to show that there is at most one blow up point for a minimizing sequence. Energy condition will be used to eliminate the case of single blow up point for the manifold not conformally equivalent to the standard sphere $(\mathbb{S}^n, g_0)$.


\section{Case of $\alpha<n$} \label{case_le}

In this section, we first prove Proposition \ref{prop-sharp-ineq}. We then analyze the sharp constant and derive Aubin type $\epsilon-$level sharp HLS inequality. Using such a sharp inequality, we finally  prove Theorem 1.3 for $\alpha<n$.

\subsection{Roughly HLS inequality on Manifolds}\label{sec roughly}

To prove Proposition \ref{prop-sharp-ineq}, we need  the following Young's inequality on manifolds.
\begin{lem}\label{lem2-1} For a given compact manifold ($M^n, g)$,  define
$$g*h(x)=\int_{M^n} g(y) h(|y-x|_g) dV_y.$$
There is a constant $C>0$, such that
$$||g*h||_{L^r} \le C||g||_{L^q} \cdot ||h||_{L^p},$$
where $p, q, r \in (1, \infty)$ and satisfy
$$1+\frac 1r =\frac 1q+\frac 1p.$$
\end{lem}
The proof is similar to the classic Young inequality in $\mathbb{R}^n$. See, e.g. Lieb and Loss \cite{LL2001}.  It is worthy of pointing out that $g*h(x)$ may not equal to $ h*g(x)$ for $x \in M^n$.


\smallskip

\noindent{\bf Proof of Proposition \ref{prop-sharp-ineq}}
The proof is quite standard. Similar proof appeared, e.g. in  Hang, Yan and Wang \cite{HWY2009} (proof of Proposition 2.1 there). To prove \eqref{1-1}, we only need to show that there is a constant $C>0$, such that for any $\lambda>0$,
\begin{equation}\label{2-1}
m\{x\in M^n \ : \ |I_{\alpha} f|>\lambda\} \le C\big( \frac{||f||_{L^p}}{\lambda}\big)^q.
\end{equation}
Inequality \eqref{1-1} follows from the above inequality via the classical Marcinkiewicz interpolation theorem.

For any $\gamma>0$, define
 \[I_{\alpha}^1 f(x)=\int_{|y-x|_g\le \gamma}\frac{f(y)}{|x-y|_g^{n-\alpha}}dy,
  \]
  and
   \[I_{\alpha}^2 f(x)=\int_{|y-x|_g> \gamma}\frac{f(y)}{|x-y|_g^{n-\alpha}}dy.
  \]
Thus, for any $\tau>0$,
\begin{equation}\label{HD-2}
m\{x: I_{\alpha} f(x)>2\tau\}\le m\{x: I_{\alpha}^1 f (x)>\tau\}+m\{x: I_{\alpha}^2 f(x)>\tau\}.
\end{equation}
 We note that  it suffices to prove  inequality \eqref{2-1} with $2\tau$ in place of $\tau$ in the left side of the
inequality, and we can further assume $\|f\|_{L^p}=1.$

From  Young inequality (Lemma \ref{lem2-1}), we have
$$
||I_{\alpha}^1 f||_{L^p} \le C \int_{|y|_g\le \gamma} \frac 1{|y|_g^{n-\alpha}} dV_g \cdot ||f||_{L^p} =c \gamma^{\alpha}.$$
Thus
$$
m\{x: I_{\alpha}^1 f (x)>\tau\} \le \frac{ ||I_{\alpha}^1 f||^p_{L^p}}{\tau^p} \le C\gamma^{p\alpha} \cdot \tau^{-p}.
$$
On the other hand,  Young inequality implies
$$||I_{\alpha}^2 f||_{L^\infty} \le C  \big(\int_{|y|_g\ge \gamma} \big(\frac 1{|y|_g^{n-\alpha}}\big)^{p'} dV_g \big)^{1/p'} \cdot ||f||_{L^p} =C_1 \gamma^{-n/q}.$$
Choose $\gamma$ so that $C_1 \gamma^{-n/q}=\tau$. Then $m\{x: I_{\alpha}^2 f (x)>\tau\}=0,$ and
$$
m\{x: I_{\alpha}^1 f (x)>\tau\} \le C\gamma^{p\alpha} \cdot \tau^{-p}=C_2 \tau^{-q}.
$$
\eqref{2-1} follows from the above easily.

For any $r\in (1, q), $  we will show the embedding is a compact. This shall be a known fact since the compact embedding is a local property, and for a bounded domain $\Omega \subset \mathbb{R}^n$, $L^r(\Omega) \subset \subset  W^{\alpha, p} (\Omega)$ is ompact, see, for example, \cite{DPV2011}. We only outline the proof here.

Let $(\Omega_i,\phi_i)_{i=1}^N$ be a finite covering of $M^n$, with each $\Omega_i$ being homeomorphic to the unite  ball $B_1(0)$  in  $\mathbb{R}^n$. Let $\{\alpha_i\}_{i=1}^N$ be a  $C^\infty$ partition of unity subordinate to the covering $\{\Omega_i\}_{i=1}^N$.

Let  $\{f_m\}_{m=1}^{\infty}$ be a bounded sequence  in $L^p(M^n)$, then for each fixed  $i=1,2, \cdots, N$, there exists a subsequence $\{\alpha_iI_\alpha f_{m_j}\}$ which is precompact in $L^r(\Omega_i)$ due to the compact embedding result on bounded domain in $\mathbb{R}^n$.

 Choosing  a  diagonal subsequence $\{I_\alpha f_{m_j}\}$, such that $\alpha_iI_\alpha f_{m_j}$ is precompact in $L^r(\Omega_i)$ for all $i=1, \cdots, N$, we then  know that   $\{I_\alpha f_{m_j}\}$ is precompact in $L^r(M^n)$,  following from Minkowski  inequality
 $$||I_\alpha f_{m_j}-I_\alpha f_{m_l}||_{L^r(M^n)}\leq\sum_{i=1}^N||\alpha_i I_\alpha f_{m_j}-\alpha_iI_\alpha f_{m_l}||_{L^r(M^n)} \to 0.$$
We hereby complete the proof of Proposition \ref{prop-sharp-ineq}.


\begin{rem}\label{rem2.2} It is quite clear that a similar augument to the above leads to: for $q$ saisfying \eqref{power}, there is a positive constant $C>0$, such that
\begin{equation}\label{1-1_*}
    ||I_{M^n, g, \alpha} f||_{L^q(M^n)} \le C ||f||_{L^p(M^n)}.
\end{equation}
holds for all  $f\in L^p({M}^n).$
Moreover, for $1\le r<q$,  operator $I_{M^n, g, \alpha} : L^p(M^n)  \to  L^r(M^n)$ is a compact  embedding.

\end{rem}

\subsection{Sharp constant and the generalized Yamabe problem}\label{sec-compare-constant}


\subsubsection{Best constant}

We first give a lower bound estimate for the optimal constant $Y_\alpha (M^n, g)$.
\begin{prop}\label{prop-compare-constant}
    $$\xi_\alpha\geq Y_\alpha (\mathbb{S}^n, g_0),$$
where
\begin{equation*}
    \xi_\alpha :=\sup_{f\in C^0(M^n)\backslash\{0\}} \frac{\left|\int_{M^n\times M^n}{f(x)f(y)}|x-y|_g^{\alpha-n}dV_g(x)dV_g(y)\right|} {\|f\|^2_{L^{2n/(n+\alpha)}(M^n)}}.
   \end{equation*}
%
\end{prop}
\begin{proof} For small positive constant $\lambda>0$, recall that  $f_\lambda(x)$ is given in \eqref{buble}.
Take
    $$\tilde{f}=\begin{cases}f_\lambda(x), &\text{ in } B_\delta(0),\\
        0, &\text{ in }\mathbb{R}^n\backslash B_\delta(0),\end{cases}$$
where $\delta>0$ is a fixed constant to be determined later. Then, for small enough $\lambda$, $\tilde{f}\in L^{2n/(n+\alpha)}(\mathbb{R}^n)$ and
\begin{align}\label{est-1}
    &\int_{\mathbb{R}^n\times\mathbb{R}^n}\tilde{f}(x)\tilde{f}(y)|x-y|^{\alpha-n}dxdy \nonumber\\
    =&\int_{\mathbb{R}^n\times\mathbb{R}^n}f_\lambda(x)f_\lambda(y)|x-y|^{\alpha-n}dxdy \nonumber\\
    &-2\int_{\mathbb{R}^n\times(\mathbb{R}^n\backslash B_\delta(0))}f_\lambda(x)f_\lambda(y)|x-y|^{\alpha-n}dxdy \nonumber\\
    &+\int_{(\mathbb{R}^n\backslash B_\delta(0))\times(\mathbb{R}^n\backslash B_\delta(0))}f_\lambda(x)f_\lambda(y)|x-y|^{\alpha-n}dxdy \nonumber\\
    =&Y_\alpha(\mathbb{S}^n,g_0)\|f_\lambda\|_{L^{\frac {2n}{n+\alpha}}(\mathbb{R}^n)}^2-\textbf{I}+\textbf{II},
\end{align}
where
\begin{align*}
    &\textbf{I}=2\int_{\mathbb{R}^n\times(\mathbb{R}^n\backslash B_\delta(0))}f_\lambda(x)f_\lambda(y)|x-y|^{\alpha-n}dxdy,\\
    &\textbf{II}=\int_{(\mathbb{R}^n\backslash B_\delta(0))\times(\mathbb{R}^n\backslash B_\delta(0))}f_\lambda(x)f_\lambda(y)|x-y|^{\alpha-n}dxdy.
\end{align*}

Note (see, e.g. \cite{Lieb1983} or \cite{Li2004}) 
$$
\int_{\mathbb{R}^n}f_\lambda(x)|x-y|^{\alpha-n}dx=Bf_\lambda^{\frac{n-\alpha}{n+\alpha}}(y),
$$
where $B=\pi^{\frac n2}\frac{\Gamma(\alpha/ 2)}{\Gamma((n+\alpha)/2)}.$
We have
\begin{align*}
    \textbf{I}&=C\int_{\mathbb{R}^n\backslash B_\delta(0)}|f_\lambda|^{\frac{2n}{n+\alpha}}dx\\
    =&C\int_\delta^{+\infty}\left(\frac\lambda{\lambda^2+r^2}\right)^n r^{n-1}dr\\
    =&C\int_{\frac\delta\lambda}^{+\infty}(1+t^2)^{-n}t^{n-1}dt=O(\frac\delta\lambda)^{-n},\quad\text{as }\lambda\rightarrow 0.
\end{align*}
On the other hand, from  HLS inequality,
we know that \text{II} can be estimated  as
\begin{align*}
    &\textbf{II}\leq Y_\alpha(\mathbb{S}^n,g_0)\|f_\lambda\|_{L^{2n/(n+\alpha)}(\mathbb{R}^n\backslash B_\delta(0))}^2\leq C(\frac\delta\lambda)^{-(n+\alpha)}.
\end{align*}
So, for small enough  $\lambda$,
\begin{equation}\label{est-2}
    \frac{\int_{\mathbb{R}^n\times\mathbb{R}^n}\tilde{f}(x)\tilde{f}(y)|x-y|^{\alpha-n}dxdy}{\|\tilde{f}\|_{L^{2n/(n+\alpha)}(\mathbb{R}^n)}^2}\geq Y_\alpha(\mathbb{S}^n,g_0)-C(\frac\delta\lambda)^{-n}.
\end{equation}

For any given point $P\in M^n$, choose a neighbourhood $\Omega_P\subset M^n$ so that for $\delta>0$ small enough, in a normal coordinate,  $\exp(B_{\delta})\subset\Omega_P$ and
$$(1-\epsilon)I\leq g(x)\leq(1+\epsilon)I,\quad \forall x\in B_{\delta}.$$
Thus,
$$(1-\epsilon)|x-y|\leq|x-y|_g\leq(1+\epsilon)|x-y|,  \ \  \quad \forall x, \ y \in B_{\delta}.$$

In the normal coordinates with respect to the center $P\in M^n$, let
$$v(x)=\left\{\begin{array}{lcl}f_\lambda(\exp^{-1}(x)), & in & \exp(B_{\delta})\\
0, & in & M^n\backslash\exp(B_{\delta}). \end{array}\right.$$
Then
\begin{align}
    &\int_{M^n}|v|^{2n/(n+\alpha)}dV\leq(1+\epsilon)^{\frac n2}\int_{B_\delta(0)}|f_\lambda(x)|^{2n/(n+\alpha)}dx, \nonumber\\
    &\int_{M^n}\int_{M^n}\frac{v(x)v(y)}{|x-y|_g^{n-\alpha}}dV_xdV_y =\int_{B_{\delta}(0)}\int_{B_{\delta}(0)}\frac{v(x)v(y)}{|x-y|_g^{n-\alpha}}\sqrt{ det g(x)}\sqrt{delt (y)}dxdy \nonumber\\
    &\geq \int_{B_{\delta}(0)}\int_{B_{\delta}(0)}\frac{f_\lambda(x)f_\lambda(y)} {(1+\epsilon)^{n-\alpha}|x-y|^{n-\alpha}}(1-\epsilon)^ndxdy \nonumber\\
    &=\frac{(1-\epsilon)^n}{(1+\epsilon)^{n-\alpha}}\int_{B_{\delta}(0)}\int_{B_{\delta}(0)}\frac{f_\lambda(x)f_\lambda(y)} {|x-y|^{n-\alpha}}dxdy.
 \label{est-3}
\end{align}
Thus
\begin{align*}
    \xi_\alpha&\geq \frac{\int_{M^n}\int_{M^n}{v(x)v(y)}|x-y|_g^{\alpha-n}dV_xdV_y}{\|v\|_{L^{2n/(n+\alpha)}(M^n)}^2}\\ 
    &\geq\frac{\frac{(1-\epsilon)^{n}}{(1+\epsilon)^{n-\alpha}}\int_{B_{\delta}(0)}\int_{B_{\delta}(0)}f_\lambda(x)f_\lambda(y) |x-y|^{\alpha-n}dxdy}{(1+\epsilon)^{\frac{n+\alpha}2}\|f_\lambda\|_{L^{2n/(n+\alpha)}(B_\delta(0))}^2} \\
    &\geq \frac{(1-\epsilon)^{n-1}}{(1+\epsilon)^{\frac{n+\alpha}2+n-\alpha}} \left(Y_\alpha(\mathbb{S}^n,g_0)-C(\frac\delta\lambda)^{-n}\right).
\end{align*}
Sending  $\epsilon$ and  $\lambda$ to $0$, we obtain the estimate.
\end{proof}

 With a slight modification of the above proof, we have
 \begin{crl} For $\alpha<n,$
  $$Y_\alpha(M^n, g)\geq Y_\alpha (\mathbb{S}^n, g_0).$$
 \end{crl}

Similar to Aubin's approach for solving Yamabe problem, we will establish an $\epsilon$-level sharp Hardy-Littlewood-Sobolev inequality for solving general curvature equations.

For $\alpha \in (0, n)$, $p>1$ and $q$ satisfying \eqref{power},
define
$$
N_{\alpha ,p}=\sup_{f\in L^p(\mathbb{R}^n)\backslash\{0\}} \frac{\left|\int_{\mathbb{R}^n\times \mathbb{R}^n}{f(x)f(y)}|x-y|^{\alpha-n}dxdy\right|} {\|f\|^2_{L^{p}(\mathbb{R}^n)}}.
$$

\begin{prop}[$\epsilon$-Level Inequality]\label{epsilon_ineq_prop} \ For $\alpha \in (0, n)$, $p>1$,
let $q$ be given by \eqref{power}.  For any given $\epsilon>0$, there is a constant $C(\epsilon)>0,$ such that
\begin{equation}\label{epsilon ineq}
\|I_{\alpha} f\|_{L^q(M^n)}^p\leq (N_{\alpha ,p}+\epsilon)^p\|f\|_{L^p(M^n)}^p+C(\epsilon)\|I_{\alpha+1}f\|_{L^q(M^n)}^p
\end{equation}
holds for all $f \in {L^p(M^n)}.$
\end{prop}

\begin{proof}

We only need to prove \eqref{epsilon ineq} 
for nonnegative function $f\in C(M^n).$

For fixed $\epsilon>0$, let $\{\eta_{i,\epsilon}\}_{i=1}^k$ be a partition of the unit covering, such that $0\le \eta_{i,\epsilon}\le 1$ for all $i=1, \cdots, k$ and $\sum_{i=1}^k \eta_{i,\epsilon}^p=1,$ and for all $i=1, \cdots, k,$
\begin{equation}\label{2-3}
||I_{\alpha} (\eta_{i,\epsilon}f)||_{L^q(\text{supp}\{\eta_{i,\epsilon}\})} \le (N_{\alpha ,p}+\epsilon)|| \eta_{i,\epsilon}f||_{L^p(\text{supp}\{\eta_{i,\epsilon}\})}.
\end{equation}
Thus
\begin{eqnarray}
&&||I_{\alpha} f||_{L^q(M^n)}^p
= ||(I_{\alpha} f)^p||_{L^{q/p}(M^n)}\nonumber\\
&=&||\sum_{i=1}^k  \eta_{i,\epsilon}^p (I_{\alpha} f)^p||_{L^{q/p}(M^n)}
\leq\sum_{i=1}^k  || \eta_{i,\epsilon}^p (I_{\alpha} f)^p||_{L^{q/p}(\text{supp}\{\eta_{i,\epsilon}\})}\nonumber\\
&=&\sum_{i=1}^k  ||\eta_{i,\epsilon} I_{\alpha} f||_{L^{q}(\text{supp}\{\eta_{i,\epsilon}\})}^p\nonumber\\
&\le &\sum_{i=1}^k \left( || I_{\alpha} (\eta_{i,\epsilon} f)||_{L^{q}(\text{supp}\{\eta_{i,\epsilon}\})}\right.\nonumber\\
& &\hspace{1.0cm}\left. + || \eta_{i,\epsilon} I_{\alpha}  f- I_{\alpha} (\eta_{i,\epsilon} f)||_{L^{q}(\text{supp}\{\eta_{i,\epsilon}\})}\right)^p.
\label{2-4}
\end{eqnarray}

For fixed $i$,
\begin{align}\label{formula error}
|| \eta_{i,\epsilon} I_{\alpha}  f-& I_{\alpha} (\eta_{i,\epsilon} f)||_{L^{q}(\text{supp}\{\eta_{i,\epsilon}\})}^q\nonumber\\
=& \int_{\text{supp}\{\eta_{i,\epsilon}\}} |\int_{M^n}{[\eta_{i,\epsilon}(x)- \eta_{i,\epsilon}(y)]f(y)}|x-y|_g^{\alpha-n}  dV_y|^q dV_x \nonumber\\
\le& (\max|\nabla \eta_{i,\epsilon}|)\cdot \int_{\text{supp}\{\eta_{i,\epsilon}\}} |\int_{M^n}{f(y) \text{supp}{\eta_{i,\epsilon}}|x-y|_g}|x-y|_g^{\alpha-n}  dV_y|^q dV_x\nonumber\\
\le& C(\max|\nabla \eta_{i,\epsilon}|)\cdot(1+\epsilon)^q ||I_{\alpha+1}f||_{L^{q}(M^n)}^q.
\end{align}
Bringing \eqref{2-3} and \eqref{formula error} into \eqref{2-4}, we have
\begin{align*}
||I_{\alpha} f||^p_{L^q(M^n)} &\le  \sum_{i=1}^k \big [(N_{\alpha, p}+\epsilon)||\eta_{i,\epsilon}f||_{L^p(\text{supp}\{\eta_{i,\epsilon}\})}+C ||I_{\alpha+1}f||_{L^{q}(M^n)} \big]^p\\
&\le \sum_{i=1}^k  (N_{\alpha, p}+\epsilon)^p\cdot(1+\epsilon)|| \eta_{i,\epsilon}f||_{L^p(M^n)}^p+C(\epsilon) ||I_{\alpha+1}f||_{L^{q}(M^n)}^p \\
&= (N_{\alpha, p}+\epsilon)^p\cdot(1+\epsilon)||f||^p_{L^p(M^n)}+C(\epsilon) ||I_{\alpha+1}f||_{L^{q}(M^n)}^p.
\end{align*}
\end{proof}

It is obvious that a similar $\epsilon$-level inequality also holds for operator $ I_{M^n, g, \alpha}$.
\begin{crl}\ For $\alpha \in (0, n)$, $p>1$,
let $q$ be given by \eqref{power}.  For any given $\epsilon>0$, there is a constant $C(\epsilon)>0,$ such that
\begin{equation}\label{epsilon ineq-1}
\| I_{M^n, g, \alpha} f\|_{L^q(M^n)}^p\leq (N_{\alpha ,p}+\epsilon)^p\|f\|_{L^p(M^n)}^p+C(\epsilon)\| I_{M^n, g, \alpha+1}f\|_{L^q(M^n)}^p
\end{equation}
holds for all $f \in {L^p(M^n)}.$
\label{cor_e}
\end{crl}

Based on the $\epsilon-$ level sharp HLS inequality, we can establish the criterior for the existence of maximizer to the following  quotient energy.

\begin{prop}\label{prop-criterior}
If  $$  \xi_{\alpha,p} :=\sup_{f\in L^p(M^n)\backslash\{0\}} \frac{\left|\int_{M^n\times M^n}{f(x)f(y)}|x-y|_g^{\alpha-n}dV_g(x)dV_g(y)\right|} {\|f\|^2_{L^{p}(M^n)}}>N_{\alpha ,p}, $$
 then the supremum $\xi_{\alpha, p}$ is attained.
\end{prop}

\begin{proof}{

Let $q$ be given by \eqref{power}. Choose  a maximizing sequence $\{f_i\}_{i=1}^{+\infty}\subset L^p(M^n)$ such that $\|I_{\alpha}f_i\|_{L^q(M^n)}=1$. Without loss of generality, we can also  assume that $f_i \ge 0$.

{Claim:} there exists a subsequence (still denoted as $\{f_i\}$) and $f_*\in L^p(M^n)$ such that
    $$\begin{array}{ccc}
        f_i\rightharpoonup f_*&\text{weakly in}&L^p(M^n),\\
        I_{\alpha}f_i\rightharpoonup I_{\alpha}f_*&\text{weakly in}&L^q(M^n),\\
        I_{\alpha+1}f_i\rightarrow I_{\alpha+1}f_*&\text{strongly in}&L^q(M^n).
    \end{array}$$
In fact, from  H\"older inequality, we know   $\|f_i\|_{L^p}(M^n)\le  \frac 1{\xi_\alpha}+1$ for large $i$, thus  $\{f_i\}$ is a bounded sequence in $L^p(M^n)$.  So, there exists a subsequence (still denoted as $\{f_i\}$) and a function $f_*\in L^p(M^n)$ such that
    $$f_i\rightharpoonup f_*\quad\text{weakly in}\quad L^p(M^n).$$
Meanwhile, for any $g\in L^{q'}(M^n)$, we have $I_\alpha g\in L^{p'}(M^n)$ and
    $$\|I_{\alpha+1}g\|_{L^{p'}(M^n)}\leq  C(M^n)\|I_{\alpha}g\|_{L^{p'}(M^n)}<+\infty,$$
where $q'$ is the conjugate of $q$ and $p'$ is the conjugate of $p$.
So,
    $$<I_\alpha f_i-I_\alpha f_*, g>=<f_i-f_*, I_\alpha g>\rightarrow 0\quad\text{ as }\quad i\rightarrow +\infty,$$
and
    $$<I_{\alpha+1} f_i-I_{\alpha+1} f_*, g>=<f_i-f_*, I_{\alpha+1} g>\rightarrow 0\quad\text{ as }\quad i\rightarrow +\infty.$$
Combining the compactness of $\{I_{\alpha+1}f_i\}$ concluded from  Proposition \ref{prop-sharp-ineq}, we have
    $$I_{\alpha+1}f_i\rightarrow I_{\alpha+1}f_*\quad\text{strongly in}\quad L^q(M^n).$$

Applying Brezis-Lieb Lemma \cite{BL1983}, we have
\begin{align*}
&\|f_i\|_{L^p(M^n)}^p-\|f_i-f_*\|_{L^p(M^n)}^p-\|f_*\|_{L^p(M^n)}^p=o(1),\\
&1-\|I_{\alpha}f_i-I_{\alpha}f_*\|_{L^q(M^n)}^q-\|I_{\alpha}f_*\|_{L^q(M^n)}=o(1).
\end{align*}
Also note:
$$
{\xi_{\alpha, p}}=\frac{(I_\alpha f_i, f_i)} {\|f_i\|^2_{L^{p}(M^n)}} +o(1)\le \frac 1{ \|f_i\|_{L^p(M^n)}}+o(1).
$$
Thus,
\begin{align*}
\frac 1{\xi_{\alpha, p}^p}\ge &\|f_i\|_{L^p(M^n)}^p+o(1)\\
=&\|f_i-f_*\|_{L^p(M^n)}^p+\|f_*\|_{L^p(M^n)}^p+o(1) \\
\geq &\frac 1{(N_{\alpha ,p}+\epsilon)^p}\|I_{\alpha}f_i-I_{\alpha}f_*\|_{L^q(M^n)}^p+\frac 1{\xi_{\alpha, p}^p}\|I_{\alpha}f_*\|_{L^q(M^n)}^p\\
&-\frac{C(\epsilon)}{(N_{\alpha ,p}+\epsilon)^p}\|I_{\alpha+1}f_i-I_{\alpha+1}f_*\|_{L^q(M^n)}^p+o(1)  \quad \text{(by $\epsilon$-level inequality (\ref{epsilon ineq}))}\\
=&\left(\frac 1{(N_{\alpha ,p}+\epsilon)^p}-\frac 1{\xi_{\alpha, p}^p}\right)\|I_{\alpha}f_i-I_{\alpha}f_*\|_{L^q(M^n)}^p\\
&+\frac 1{\xi_{\alpha, p}^p}\left(\|I_{\alpha}f_i-I_{\alpha}f_*\|_{L^q(M^n)}^p+\|I_{\alpha}f_*\|_{L^q(M^n)}^p\right)\\
&-\frac{C(\epsilon)}{(N_{\alpha ,p}+\epsilon)^p}\|I_{\alpha+1}f_i-I_{\alpha+1}f_*\|_{L^q(M^n)}^p+o(1)\\
\geq &\left(\frac 1{N_{\alpha ,p}+\epsilon)^p}-\frac 1{\xi_{\alpha, p}^p}\right)\|I_{\alpha}f_i-I_{\alpha}f_*\|_{L^q(M^n)}^p\\
&+\frac 1{\xi_{\alpha, p}^p}\left(\|I_{\alpha}f_i-I_{\alpha}f_*\|_{L^q(M^n)}^q+\|I_{\alpha}f_*\|_{L^q(M^n)}^q\right)+o(1)\\
=&\left(\frac 1{(N_{\alpha ,p}+\epsilon)^p}-\frac 1{\xi_{\alpha, p}^p}\right)\|I_{\alpha}f_i-I_{\alpha}f_*\|_{L^q(M^n)}^p +\frac 1{\xi_{\alpha, p}^p}+o(1).
\end{align*}
So
$$\lim_{i\rightarrow+\infty}\|I_{\alpha}f_i-I_{\alpha}f\|_{L^q(M^n)}=0.$$
On the other hand,
$$\|f_*\|_{L^p(M^n)}\le \liminf_{i\rightarrow+\infty}\|f_i\|_{L^p(M^n)},$$
we thus know
$$
\lim_{i \to \infty} {\frac{(I_\alpha f_i, f_i)} {\|f_i\|^2_{L^{p}(M^n)}}} \le \frac{(I_\alpha f_*, f_*)} {\|f_*\|^2_{L^{p}(M^n)}}
$$
So $f_*\in L^p(M^n)$ is a maximizer.
}
\end{proof}

Similarly, based on Corollary \ref{cor_e}, we can obtain the following
\begin{crl}If   $$  \xi_{\alpha,p, G} :=\sup_{f\in L^p(M^n)\backslash\{0\}} \frac{\left|\int_{M^n\times M^n}{f(x)f(y)}[G^g_x(y)]^{\frac{\alpha -n}{2-n}}dV_g(x)dV_g(y)\right|} {\|f\|^2_{L^{p}(M^n)}}>N_{\alpha ,p}, $$
 then the supremum $\xi_{\alpha, p, G}$ is attained.
\label{cor_exist}
\end{crl}



%
%

\subsubsection{Genaralized Yamabe problem}

%

We shall prove Theorem \ref{yamabe} for $\alpha<n$  in this subsection. Due to Corollary \ref{cor_exist}, we only need to prove
\begin{prop}
If $(M^n, g)$ is locally conformally flat, but  not conformally equivalent to the standard sphere $(\mathbb{S}^n ,g_0)$, then for $\alpha<n$, $Y_\alpha (M^n, g)>Y_\alpha (\mathbb{S}^n, g_0).$
\label{Yamabe}
\end{prop}


%
%
%
%
%
%
%

From now on in this subsection, we will assume that  $(M^n, g)$ is a locally conformally flat manifold. We need the follow expansion for  Green's function of conformal Laplacian operator near its singular point (Lemma 6.4 in \cite{LP1987}, here we use the same notations).
\begin{lem} Let $(M^n, g)$ be a locally conformally flat manifold ($n\ne 2$).
In conformal normal coordinates $\{x^i\}$ at $x$, $G_x^g(y)$  has an asymptotic expansion
\begin{equation}\label{11-17-1}
G_x^g(y)=r^{2-n}+A+O''(r),  \ \ \   \  \forall   \  y \in B_{\delta_0}(x)
\end{equation}
where $A \ge 0$ is a constant.\label {expansion of Green function}
\end{lem}

\smallskip

\noindent{\bf Proof of Proposition \ref{Yamabe}}.
Let  $P\in M^n$ be a fixed point. In a conformal normal coordinate around $P$, $G_x^g(y)$ satisfies \eqref{11-17-1}. Further, since the manifold is not conformally equivalent to the standard sphere, $A>0$ by the positive mass theorem. For simplicity, we denote $B_{\delta}(P)$ as $B_{\delta}$.


For small enough $\delta>0$,  choose a  test function
\begin{equation}\label{local flat 1}
    u=\left\{\begin{array}{ll}
    f_\lambda(x), & B_{\delta},\\
    0, & M^n\backslash B_{\delta}.
    \end{array}\right.
\end{equation}
Then by a similar argument to  { Proposition \ref{prop-compare-constant}}, we can obtain
\begin{equation}\label{local flat 2}
    \begin{split}
    Y_\alpha(M^n,g)\geq &J_{g,\alpha}(u)\\
    \geq & Y_\alpha(\mathbb{S}^n,g_0)-C\left(\frac\delta\lambda\right)^{-n}\\ &+A\cdot \frac{\int_{B_\delta\times B_\delta}|x-y|_g^{\alpha-2}f_\lambda(x)f_\lambda(y)dxdy} {\|u\|_{L^{\frac{2n}{n+\alpha}}(M^n)}^2}.
    \end{split}
\end{equation}
Since
\begin{equation}\label{local flat 3}
    \begin{split}
    &\int_{B_\delta\times B_\delta}|x-y|_g^{\alpha-2}f_\lambda(x)f_\lambda(y)dxdy\\ =&\lambda^{-(n+\alpha)}\int_{B_\delta\times B_\delta}|x-y|_g^{\alpha-2} \left(1+\frac{|x|^2}{\lambda^2}\right)^{-\frac{n+\alpha}2} \left(1+\frac{|y|^2}{\lambda^2}\right)^{-\frac{n+\alpha}2}dxdy\\
    =&\lambda^{n-2}\int_{B_{\delta/\lambda}\times B_{\delta/\lambda}}|u-v|^{\alpha-2} (1+|u|^2)^{-\frac{n+\alpha}2} (1+|v|^2)^{-\frac{n+\alpha}2} dudv\\
    \geq&C_1\lambda^{n-2},
    \end{split}
\end{equation}
thus
\begin{equation}\label{local flat 4}
    \begin{split}
    &-C\left(\frac\delta\lambda\right)^{-n}+A \cdot \frac{\int_{B_\delta\times B_\delta}|x-y|_g^{\alpha-2}f_\lambda(x)f_\lambda(y)dxdy} {\|u\|_{L^{\frac{2n}{n+\alpha}}(M^n)}^2}\\
    \geq&-C\left(\frac\delta\lambda\right)^{-n}+C_2A\lambda^{n-2} =\lambda^{n-2}(C_2A-C\lambda^2\delta^{-n})>0
    \end{split}
\end{equation}
by choosing $\lambda$ much smaller than $\delta$. We hereby obtain
\begin{equation}\label{local flat 5}
    Y_\alpha(M^n,g)>Y_\alpha(\mathbb{S}^n,g_0).
\end{equation}

\section{Case of $\alpha>n$}\label{case_ge}

We first establish the reversed HLS inequality on general compact Riemannian manifolds.

\subsection{Reversed  HLS inequality on Manifolds}



We need the follwing two lemmas.

\begin{lem}[Conversed Young's Inequality]\label{lem2-1-1}
There is a constant $C>0$, such that
$$||g*h||_{L^r} \ge C||g||_{L^q} \cdot ||h||_{L^p}$$
where $p\in (0,1), q, r<0$ and satisfying
$$1+\frac 1r =\frac 1q+\frac 1p.$$
\end{lem}

This can be proved in a simiar way to that for   Conversed Young's Inequality in $\mathbb{R}^n$. We skip details here.

For a given measurable function $f(x)$ on $M^n$ and $0<p<+\infty$, the weak $L^p$ norm of $f(x)$ is defined by
    $$\|f\|_{L_W^p}=\inf\{A>0:\ meas\{|f(x)|>t\}\cdot t^p\leq A^p\},$$
For $p<0$, the norm is defined as
    $$\|f\|_{L_W^p}=\sup\{A>0:\ meas\{|f(x)|<t\}\cdot t^p\leq A^p\}.$$
Thus, for $p<0$,
    $$\|f\|_{L_W^p}^p=\inf\{B>0:\ meas\{|f(x)|<t\}\cdot t^p\leq B\}.$$

Let $T:L^p(M^n)\rightarrow L^q(M^n)$ be a linear operator. We recall that for $0<p,q<+\infty$, operator $T$ is called the weak type $(p,q)$ if there exists a constant $C(p,q)>0$ such that for all $f\in L^p(M^n)$
    $$meas\{x: |Tf(x)|>\tau\}\leq\left(C(p,q)\frac{\|f\|_{L^p}}\tau \right)^q, \quad \forall\tau>0.$$
For $q<0<p<1$, we say operator $T$ is of the weak type $(p,q)$, if there exists a constant $C(p,q)>0$, such that for all $f\in L^p(M^n)$,
    $$meas\{x: |Tf(x)|<\tau\}\leq\left(C(p,q)\frac{\|f\|_{L^p}}\tau \right)^q, \quad \forall\tau>0.$$

\begin{lem}[Marcinkiewicz type interpolation theorem]\label{lem Marcinkiewicz}
Let $T$ be a linear operator which maps any nonnegative function to a nonnegative function. For a pair of numbers $(p_1,q_1),\ (p_2,q_2)$ satisfying $q_i<0<p_i<1,\ i=1,2,\ p_1<p_2$ and $q_1<q_2$, if $T$ is weak type $(p_1,q_1)$ and $(p_2,q_2)$ for all nonnegative functions, then for any $\theta\in(0,1)$, and
    $$\frac 1p=\frac{1-\theta}{p_1}+\frac\theta{p_2},\quad \frac 1q=\frac{1-\theta}{q_1}+\frac\theta{q_2},$$
$T$ is reversed strong type $(p,q)$ for all nonnegative functions, that is,
    $$\|Tf\|_{L^q}\geq C\|f\|_{L^p},\quad \forall f\in L^p\quad\text{and}\quad f\geq 0,$$
for some constant $C=C(p_1,p_2,q_1,q_2)>0$.
\end{lem}

The proof of Lemma \ref{lem Marcinkiewicz} is almost identical to that for the same inequality in $\mathbb{R}^n$, see  Dou and Zhu \cite{DZ2}.

\smallskip

{\bf Proof of Theorem \ref{prop-rough-ineq_R}.} The proof is quite standard. We shall follow the proof of reversed Hardy-Littlewood-Sobolev inequality, given in Dou and Zhu \cite{DZ2} (proof of Proposition 2.3 there). To prove \eqref{1-1-1}, we only need to show that there is a constant $C>0$, such that for any $\lambda>0$,
\begin{equation}\label{2-1-1}
m\{x\in M^n \ : \ |I_{\alpha} f|<\lambda\} \le C\big( \frac{||f||_{L^p}}{\lambda}\big)^q.
\end{equation}
Inequality \eqref{1-1-1} follows from the above inequality via the above Marcinkiewicz type interpolation theorem (Lemma \ref{lem Marcinkiewicz}).

For any $\gamma>0$, define
 \[I_{\alpha}^1 f(x)=\int_{|y-x|_g\le \gamma}\frac{f(y)}{|x-y|^{n-\alpha}}dy,
  \]
  and
   \[I_{\alpha}^2 f(x)=\int_{|y-x|_g> \gamma}\frac{f(y)}{|x-y|^{n-\alpha}}dy.
  \]
Thus, for any $\tau>0$,
\begin{equation}\label{HD-2-1}
    m\{x: I_{\alpha} f(x)<2\tau\}\le m\{x: I_{\alpha}^1 f (x)<\tau\}+m\{x: I_{\alpha}^2 f(x)<\tau\}.
\end{equation}
 We note that  it suffices to prove  inequality \eqref{2-1-1} with $2\tau$ in place of $\tau$ in the left side of the
inequality, and we can further assume $\|f\|_{L^p}=1.$
From  Conversed Young's inequality (Lemma \ref{lem2-1-1}), we have
    $$||I_{\alpha}^1 f||_{L^{r_1}} \ge C \left(\int_{M^n}\left(\frac{\chi_\gamma(y)}{|y|^{n-\alpha}} \right)^{t_1}dy\right)^{\frac 1{t_1}} ||f||_{L^p} =:D_1,$$
where $\frac 1p+\frac 1{t_1}=1+\frac 1{r_1}$ with $t_1\in(\frac n{n-\alpha},0),\ r_1<0$, $\chi_\gamma(x)=1$ for $|x|_g\leq\gamma$ and $\chi_\gamma(x)=0$ for $|x|_g>\gamma$, and
    $$D_1=\left(\int_{M^n}\left(\frac{\chi_\gamma(y)} {|y|^{n-\alpha}} \right)^{t_1}dy\right)^{\frac 1{t_1}}=C_1(n,\alpha)\gamma^{\frac{n-(n-\alpha)t_1}{t_1}}.$$
Thus
\begin{equation}\label{2-5}
    m\{x: I_{\alpha}^1 f (x)<\tau\} \le \frac{ ||I_{\alpha}^1 f||^{r_1}_{L^{r_1}}}{\tau^{r_1}} \le \frac{C_2(n,\alpha)\gamma^{\frac{r_1[n-(n-\alpha)t_1]}{t_1}}} {\tau^{r_1}}.
\end{equation}
On the other hand, Conversed Young's inequality implies
    $$||I_{\alpha}^2 f||_{L^{r_2}} \ge C \left(\int_{M^n}\left(\frac{1-\chi_\gamma(y)}{|y|^{n-\alpha}} \right)^{t_2}dy\right)^{\frac 1{t_2}} ||f||_{L^p} =:D_2,$$
where $\frac 1p+\frac 1{t_2}=1+\frac 1{r_2}$ with $t_2<\frac n{n-\alpha},\ r_2<0$ and
    $$D_2=\left(\int_{M^n}\left(\frac{1-\chi_\gamma(y)} {|y|^{n-\alpha}} \right)^{t_2}dy\right)^{\frac 1{t_2}}=C_3(n,\alpha)\gamma^{\frac{n-(n-\alpha)t_2}{t_2}}.$$
It  follows that $r_1<\frac{np}{n-\alpha p}<r_2$ and
\begin{equation}\label{2-6}
    m\{x: I_{\alpha}^2 f (x)<\tau\} \le \frac{ ||I_{\alpha}^1 f||^{r_2}_{L^{r_2}}}{\tau^{r_2}} \le \frac{C_4(n,\alpha)\gamma^{\frac{r_2[n-(n-\alpha)t_2]}{t_2}}} {\tau^{r_2}}.
\end{equation}
Bringing \eqref{2-5} and \eqref{2-6} into \eqref{HD-2-1}, we have
    $$m\{x: I_{\alpha} f(x)<2\tau\}\le \frac{C_2(n,\alpha)\gamma^{\frac{r_1[n-(n-\alpha)t_1]}{t_1}}} {\tau^{r_1}}+\frac{C_4(n,\alpha)\gamma^{\frac{r_2[n-(n-\alpha)t_2]}{t_2}}} {\tau^{r_2}}.$$

Choose $\gamma=\tau^{\frac p{p\alpha-n}}$. We have
\begin{equation*}
    \frac{pr_1}{p\alpha-n}\left[\frac n{t_1}-(n-\alpha)\right]-r_1 =\frac{pr_2}{p\alpha-n}\left[\frac n{t_2}-(n-\alpha)\right]-r_2=-\frac{np}{n-p\alpha}=-q.
\end{equation*}
 We thus  obtain \eqref{2-1-1} and complete the proof of Theorem \ref{prop-rough-ineq_R}.

\subsection{Sharp constant and the generalized Yamabe proble for $\alpha>n$}\label{sec-compare-constant-1}

\subsubsection{Best constant}

Recall, for $\alpha>n$,
\begin{equation*}
  Y_\alpha (M^n, g)=\inf_{f\in L^{2n/(n+\alpha)}(M^n)\backslash\{0\}} \frac{\left|\int_{M^n\times M^n}{f(x)f(y)}[G_x^g(y)]^{\frac{\alpha-n}{2-n}}dV_g(x)dV_g(y)\right|} {\|f\|^2_{L^{2n/(n+\alpha)}(M^n)}}.
\end{equation*}
It follows from Dou-Zhu's result \cite{DZ2}, that the infimum on  the standard sphere or flat plane is achieved, $Y_\alpha(\mathbb{S}^n, g)$ is  given by \eqref{8-5-1} and extremal functions  on flat plane $\mathbb{R}^n $ are given by \eqref{buble}.

We first give an upper bound estimate for the optimal constant $Y_\alpha (M^n, g)$.
\begin{prop}\label{prop-compare-constant-1}
If $\alpha>n$, then
    $$Y_\alpha (M^n, g) \leq Y_\alpha(\mathbb{S}^n, g_0).$$

\end{prop}
\begin{proof}
Take
    $$\tilde{f}=\begin{cases}f_\lambda(x), &\text{ in } B_\delta(0),\\
        0, &\text{ in }\mathbb{R}^n\backslash B_\delta(0),\end{cases}$$
where $\delta>0$ is a fixed constant to be determined later. Then, for small enough $\lambda$, $\tilde{f}\in L^p(\mathbb{R}^n)$ and
\begin{align}\label{est-1-1}
    &\int_{\mathbb{R}^n\times\mathbb{R}^n}\tilde{f}(x)\tilde{f}(y)|x-y|^{\alpha-n}dxdy \nonumber\\
    =&\int_{\mathbb{R}^n\times\mathbb{R}^n}f_\lambda(x)f_\lambda(y)|x-y|^{\alpha-n}dxdy \nonumber\\
    &-\int_{\mathbb{R}^n\times(\mathbb{R}^n\backslash B_\delta(0))}f_\lambda(x)f_\lambda(y)|x-y|^{\alpha-n}dxdy \nonumber\\
    &-\int_{B_\delta(0)\times(\mathbb{R}^n\backslash B_\delta(0))}f_\lambda(x)f_\lambda(y)|x-y|^{\alpha-n}dxdy \nonumber\\
    =&Y_\alpha(\mathbb{S}^n, g_0)\|f_\lambda\|_{L^{2n/(n+\alpha)}(\mathbb{R}^n)}^2 -\textbf{I}-\textbf{II},
\end{align}
where
\begin{align*}
    &\textbf{I}=\int_{\mathbb{R}^n\times(\mathbb{R}^n\backslash B_\delta(0))}f_\lambda(x)f_\lambda(y)|x-y|^{\alpha-n}dxdy,\\
    &\textbf{II}=\int_{B_\delta(0)\times(\mathbb{R}^n\backslash B_\delta(0))}f_\lambda(x)f_\lambda(y)|x-y|^{\alpha-n}dxdy.
\end{align*}

Note (see \cite{DZ2})
$$
\int_{\mathbb{R}^n}f_\lambda(x)|x-y|^{\alpha-n}dx=Bf_\lambda^{\frac{n-\alpha}{n+\alpha}}(y),
$$
where $B=\pi^{\frac n2}\frac{\Gamma(\alpha/ 2)}{\Gamma((n+\alpha)/2)}.$
We have
\begin{align*}
    \textbf{I}=&C\int_{\mathbb{R}^n\backslash B_\delta(0)}|f_\lambda|^{2n/(n+\alpha)}dx\\
    =&C\int_\delta^{+\infty}\left(\frac\lambda{\lambda^2+r^2}\right)^n r^{n-1}dr\\
    =&C\int_{\frac\delta\lambda}^{+\infty}(1+t^2)^{-n}t^{n-1}dt=O(\frac\delta\lambda)^{-n},\quad\text{as }\lambda\rightarrow 0.
\end{align*}
On the other hand, from the reversed HLS inequality, we know that \text{II} can be estimated  as
\begin{align*}
    &\textbf{II}\geq C \|f_\lambda\|_{L^{2n/(n+\alpha)}(B_\delta(0)}\|f_\lambda\|_{L^{2n/(n+\alpha)}(\mathbb{R}^n\backslash B_\delta(0))}=O(\frac\delta\lambda)^{-\frac{n+\alpha}2}\quad\text{as }\lambda\rightarrow 0.
\end{align*}

Note that  $\frac{n+\alpha}n>1$.
We have
\begin{equation*}
    \begin{split}
    \|f_\lambda\|_{L^{2n/(n+\alpha)}(\mathbb{R}^n)}^2= &\left(\int_{B_\delta(0)}|f_\lambda|^{2n/(n+\alpha)}dx +\int_{\mathbb{R}^n\backslash B_\delta(0)}|f_\lambda|^{2n/(n+\alpha)}dx\right)^{(n+\alpha)/n}\\
    \leq &\|f_\lambda\|_{L^{2n/(n+\alpha)}(B_\delta(0))}^2 +C\|f_\lambda\|_{L^{2n/(n+\alpha)}(B_\delta(0))}^{2-p} \|f_\lambda\|_{L^{2n/(n+\alpha)}(\mathbb{R}^n\backslash B_\delta(0))}^p.
    \end{split}
\end{equation*}
So, for small enough  $\lambda$,
\begin{equation}\label{est-2-1}
    \frac{\int_{\mathbb{R}^n\times\mathbb{R}^n}\tilde{f}(x) \tilde{f}(y)|x-y|^{\alpha-n}dxdy} {\|\tilde{f}\|_{L^p(\mathbb{R}^n)}^2}\leq Y_\alpha(\mathbb{S}^n, g_0)+C(\frac\delta\lambda)^{-n}.
\end{equation}
The rest of the argument can be carried out in the same way as in the proof of Proposition \ref{prop-compare-constant}.
\end{proof}

  To prove  Theorem \ref{yamabe} for $\alpha>n$, we first prove
%
\begin{prop}
If $Y_\alpha (M^n, g)< Y_\alpha(\mathbb{S}^n, g_0)$, then the infimum is attained.
\label{criterior_R}
\end{prop}

The proof will base on a new blowup analysis.
For subcritical power  $p\in (0, \frac{2n}{n+\alpha})$, we consider
the infimum
\begin{equation}
Y_{\alpha,p} (M^n, g):=\inf_{u\in C^0(M^n)\setminus \{0\}, u\ge 0} J_{g,\alpha, p}(u).
\end{equation}
where
\begin{equation*}
    J_{g,\alpha,p}(u):=\frac{\int_{M^n}\int_{M^n}{u(x)u(y)}[G^{g}_x(y)]^{\frac {\alpha-n}{2-n}}dV_g(x)dV_g(y) }{\|u\|^2_{L^p(M^n)}}.
    \label{infimum-2}
\end{equation*}

\begin{lem}\label{subcritical}
For subcritical power  $p\in (0, \frac{2n}{n+\alpha})$. infimum $Y_{\alpha,p} (M^n, g)$ is attained.
\end{lem}


\begin{proof}The lemma could be proved via establishing certain compactness embedding for $\alpha>n$, which is not known. To circumnavigate this difficulty, we here use a new blowup type argument. The main difference between our new blowup analysis with the traditional one is that: our argument is a global one since we do not have a  local Sobolev type inequality (the classic concentratione compactness, as well as Nash-Moser iteration are not available).

For fixed $p\in (0, \frac {2 n}{n +\alpha}),$ let $u_i$ be a minimizing positive sequence of $Y_{\alpha,p} (M^n, g)$ with $\|u\|^2_{L^p(M^n)}=1$. Easy to see that, up to further subsequence, $u_i \to u_* \in {L^p(M^n)}$  pointwise.

We consider two cases:

\noindent {\bf Case 1:} There are at least two  points on $M^n$, say $x_0, \ x_1$ and a universal positive constant $C>0$,  such that,  for any $r>0$, there is a subsequence of $u_i$, satisfying
$$
 \lim_{i \to \infty }\int_{B_r(x_0)} u_i^p >C; \   \  \  \  \  \   \lim_{i \to \infty}\int_{B_r(x_1)} u_i^p >C.
$$
Note $p<1$. The above inequality implies:
\begin{equation}\label{4-10-1+1}
 \lim_{i \to \infty}\int_{B_r(x_0)} u_i >0; \   \  \  \  \  \   \lim_{i \to \infty} \int_{B_r(x_1)} u_i >0.
\end{equation}

Denote
\begin{equation}
I_{g, \alpha}u_i(x)=\int_{{M}^n} u_i(y)[G^{g}_x(y)]^{\frac {\alpha-n}{2-n}})dV_g(y).
\label{I-int+1}
\end{equation}
We then know, due to \eqref{4-10-1+1} that there is a universal positive constant $C>0$, such that
\begin{equation}
I_{g, \alpha}u_i(\xi) \ge C \ \ \ \ \mbox{for \ \ \ all} \ \ \ \xi \in M^n.
\label{4-10-2+1}
\end{equation}
On the other hand, if $meas\{\xi\in M^n \ : I_{g, \alpha} u_i(\xi) \to \infty \ \mbox{as} \  i\to \infty\}=vol(M^n),$ then we have, using \eqref{4-10-1+1}, that $H_{\alpha, R}(u_i, u_i) \to \infty$, which contradicts the assumption that $u_i$ is a minimizing sequence. Thus $I_{g, \alpha}u_i(\xi)$ stays uniformly bounded in a set with positive measure. This implies: there is a constant $C_1>0$, such that
\begin{equation}
\int_{M^n}u_i(\xi) dS_\xi \le C_1.
\label{4-10-3+1}
\end{equation}
From \eqref{4-10-3+1} we know that sequence $\{I_{g, \alpha}u_i(\xi)\}_{i=1}^\infty$ is uniformly bounded and equiv-continuous on $M^n$. Up to a subsequence, $I_{g, \alpha}u_i(x) \to L(x)\in C(M^n)$.

Using Fatou Lemma and the reversed Hardy-Littlewood-Sobolev inequality (see Dou and Zhu \cite{DZ2}), we have, up to a further subsequence,  that, for any positive integer $m>0$,
\begin{align*}
0&\ge\big(\lim_{i  \to \infty}\int_{M^n}|I_{g, \alpha}u_{i}-I_{g, \alpha}u_{i+m}|^{2n/(n-\alpha)}  \big)^{(n-\alpha)/2n}\\
&\ge C\big(\lim_{i  \to \infty}||u_{i}-u_{i+m}||^{2n/(n-\alpha)}_{L^{2n/(n+\alpha)}}\big)^{(n-\alpha)/2n}.
\end{align*}

Thus $||u_i-u_{i+m}||_{L^{2n/n+\alpha)}}\to 0$. This implies $||u_i-u_*||_{L^{2n/(n+\alpha)}}\to 0$.  Thus the infimum  is achieved by $u_*$.  Easy to see that $u_*>0$ every where on $M^n$.

We are left to rule out the following case.

\noindent{\bf Case 2}.  Single point blow up point:  There is only one point $x_0 \in M^n$, such that for any $r>0$, there is a subsequence of $u_i$, such that
\begin{equation}\label{energy_c+1}
\lim_{i \to \infty}\int_{B_r(x_0)} u_i^p=1.
\end{equation}
It follows from \eqref{energy_c+1} and H\"older inequality that
$$
 \lim_{i \to \infty}\int_{B_r(x_0)} u_i^{\frac{2n}{n+\alpha}} \to \infty.
$$
On the other hand,  $||I_{g, \alpha}u_i(x)||_{L^{\frac{2n}{n-\alpha}}}$ is bounded, thus $||u_i||_{L^{\frac {2n}{n+\alpha}}}$ is bounded via the reversed HLS inequality. Contradiction. Thus case 2 can not happen.

\end{proof}

Lemma \ref{subcritical}  yields that the infimum $Y_{\alpha,p} (M^n, g)$  is attained.  Let $u_p$ be a minimizer such that $ \|u\|_{L^p(M^n)}=1$. It can be proved that $u_p$ is smooth function (see, for example, \cite{Li2004}, or \cite{DZ1}). To complete the proof of Proposition 3.4,  we discuss two cases.

\noindent {\bf Case 1:} There are at least two  points on $M^n$, say $x_0, \ x_1$ and a universal positive constant $C>0$,  such that,  for any $r>0$, there is a subsequence of $u_p$, satisfying
$$
 \lim_{p\to \frac{2n}{n+\alpha}}\int_{B_r(x_0)} u_p^p >C; \   \  \  \  \  \   \lim_{p\to \frac{2n}{n+\alpha}}\int_{B_r(x_1)} u_p^p >C.
$$
Note $p<1$. The above inequality implies:
\begin{equation}\label{4-10-1}
 \lim_{p\to \frac{2n}{n+\alpha}}\int_{B_r(x_0)} u_p >0; \   \  \  \  \  \   \lim_{p\to \frac{2n}{n+\alpha}}\int_{B_r(x_1)} u_p >0.
\end{equation}
%
%
We then know, due to \eqref{4-10-1} that there is a universal positive constant $C_4>0$, such that
\begin{equation}
I_{g, \alpha}u_p(\xi) \ge C_4 \ \ \ \ \mbox{for \ \ \ all} \ \ \ \xi \in M^n.
\label{4-10-2}
\end{equation}
On the other hand, if $meas\{\xi\in M^n \ : I_{g, \alpha} u_p(\xi) \to \infty \ \mbox{as} \  i\to \infty\}=vol(M^n),$ then we have, using \eqref{4-10-1}, that $H_{g, \alpha}(u_p, u_p) \to \infty$, which contradicts the assumption that $u_p$ is a minimizing sequence. Thus $I_{g, \alpha}u_p(\xi)$ stays uniformly bounded in a set with positive measure. This implies: there is a constant $C_5>0$, such that
\begin{equation}
\int_{M^n}u_p(\xi) dS_\xi \le C_5.
\label{4-10-3}
\end{equation}
From \eqref{4-10-3} we know that sequence $\{I_{g, \alpha}u_p(\xi)\}_{i=1}^\infty$ is uniformly bounded and equiv-continuous on $M^n$. Up to a subsequence, $I_{\alpha, R}u_p(x) \to L(x)\in C(M^n)$.

Using Fatou Lemma and the reversed Hardy-Littlewood-Sobolev inequality (see Dou and Zhu \cite{DZ2}), we have, up to a further subsequence,  that,
\begin{align*}
0&\ge\big(\lim_{p_1, p_2 \to 2n/(n+\alpha)}\int_{M^n}|I_{g, \alpha}u_{p_1}-I_{g, \alpha}u_{p_2}|^{2n/(n-\alpha)}  \big)^{(n-\alpha)/2n}\\
&\ge C\big(\lim_{p_1, p_2 \to 2n/(n+\alpha)}||u_{p_1}-u_{p_2}||^{2n/(n-\alpha)}_{L^{2n/(n+\alpha)}}\big)^{(n-\alpha)/2n}.
\end{align*}

Thus $||u_i-u_{j}||_{L^{2n/n+\alpha)}}\to 0$. This implies $||u_i-u_\circ||_{L^{2n/(n+\alpha)}}\to 0$ for some $u_\circ$. Thus,
 up to a further subsequence, $u_i \to u_\circ \ge 0$ almost everywhere.  Dominate convergence theorem yields that $||u_\circ||_{L^{2n/(n+\alpha)}}=1.$
It follows, via Fatou Lemma, that $\lim_{i, j\to \infty}H_{\alpha, R}(u_i, u_j)
\ge H_{\alpha, R}(u_\circ, u_\circ).$
Thus the infimum  is achieved by $f_\circ \ge 0$.

%
%
%
%
%

\smallskip
Using energy condition, we will rule out

\noindent {\bf Case 2:} Single point blow up point:  There is only one point $x_0 \in M^n$, such that for any $r>0$, there is a subsequence of $u_p$, such that
\begin{equation}\label{energy_c}
\lim_{p\to \frac{2n}{n+\alpha}}\int_{B_r(x_0)} u_p^p=1.
\end{equation}
It follows \eqref{energy_c} that
$$
 \lim_{p\to \frac{2n}{n+\alpha}}\int_{B_r(x_0)} u_p >0.
$$
If there is another point $x_1 \ne x_0$, such that for small $r>0$,
$$
 \lim_{p\to \frac{2n}{n+\alpha}}\int_{B_r(x_1)} u_p >0.
$$
We then again can obtain the existence of minimizer using the above argument.

 Finally,  if for any $x\ne x_0$, , such that for small  $r< dist (x, x_0)$,
\begin{equation}\label{sep-19-1}
 \lim_{p\to \frac{2n}{n+\alpha}}\int_{B_r(x)} u_p =0,
\end{equation}
we shall show that in this case $ Y_\alpha (M^n, g)\ge Y_\alpha (\mathbb{S}^n, g_0)$ which contradicts  to the energy constraint $Y_\alpha (M^n, g)<Y_\alpha (\mathbb{S}^n, g_0)$.

In fact, from the assumption of one blowup point \eqref{energy_c} and \eqref{sep-19-1} (also notice that $\alpha>n$), we know that for small enough $r>0$,
$$
\lim_{p\to \frac{2n}{n+\alpha}} J_{g,\alpha,p}(u_p)=  \lim_{p\to \frac{2n}{n+\alpha}} \frac{\left|\int_{B_r(x_0)}\int_{B_r(x_0)}{u_p(x)u_p(y)}[G^{g}_x(y)]^{\frac {\alpha-n}{2-n}}dV_g(x)dV_g(y) \right|}{\|u_p\|^2_{L^p(B_r(x_0))}}.
$$
Since $M^n$ is locally conformally flat, we know
\begin{align*}
&  \lim_{p\to \frac{2n}{n+\alpha}} \frac{ \int_{B_r(x_0)}\int_{B_r(x_0)}{u_p(x)u_p(y)}[G^{g}_x(y)]^{\frac {\alpha-n}{2-n}}dV_g(x)dV_g(y) }{\|u_p\|^2_{L^p(B_r(x_0))}}
\\
&= \lim_{p\to \frac{2n}{n+\alpha}} \frac{\int_{B_r(x_0)}\int_{B_r(x_0)}{u_p(x)u_p(y)}  [G^{g}_x(y)]^{\frac {\alpha-n}{2-n}})dV_g(x)dV_g(y) }{\|u_p\|^2_{L^{2n/(n+\alpha)}(B_r(x_0))}}
\\
&= \lim_{p\to \frac{2n}{n+\alpha}} \frac{\int_{S_R}\int_{S_R}{u_p(x)u_p(y)}|x-y|^{ {\alpha-n}}dx d y }{\|u_p\|^2_{L^{2n/(n+\alpha)}(S_R)}}
\\
&\ge \inf_{u\in L^{2n/(n+\alpha)}(\mathbb{R}^n )\setminus\{0\}, u \ge 0 } \frac{\int_{\mathbb{R}^n}\int_{\mathbb{R}^n}{u(x)u(y)}|x-y|^{ {\alpha-n}}dx d y }{\|u\|^2_{L^{2n/(n+\alpha)}(\mathbb{R}^n)}}
\\
&=Y_\alpha(\mathbb{S}^n, g_0),
\end{align*}
where $S_R \in \mathbb{R}^n$ is the image of $B_r(x_0)\in M^n$ under a conformal map from a local chart containingin  $B_r(x_0)$ to $ \mathbb{R}^n$.

We hereby complete the proof of Proposition \ref{criterior_R}.

\medskip

To  complete the proof  Theorem \ref{yamabe} for $\alpha>n$, we  are  left to show

\begin{prop} If $(M^n, g)$ is locally conformally flat, but  not conformally equivalent to the standard sphere $(\mathbb{S}^n ,g_0)$, then for $\alpha>n$,  $Y_\alpha (M^n, g)<Y_\alpha (\mathbb{S}^n, g_0).$
\label{conformal_flat}
\end{prop}
%
%

\begin{proof}
Let  $P\in M^n$ be a fixed point. In a conformal normal coordinate around $P$, $G_x^g(y)$ satisfies \eqref{11-17-1}. Further, since the manifold is not conformally equivalent to the standard sphere, $A>0$ by the positive mass theorem.

Since  $\alpha>n>2$, we know that there exist two positive constants $\delta_0, A_0$ such that
    $$(G_x(y))^{\frac{\alpha-n}{2-n}}\leq |x-y|_g^{\alpha-n}-A_0|x-y|_g^{\alpha-2},\quad \quad \forall\ x,y\in B_{\delta_0}(P).$$
In the sequel of the proof, we denote $B_{\delta}(P)$ as $B_{\delta}$.

For any fixed $\delta\in (0,\delta_0)$, take a specific test function as
$$
    u=\left\{\begin{array}{ll}
    f_\lambda(x), & B_{\delta},\\
    0, & M^n\backslash B_{\delta}.
    \end{array}\right.
$$
Similar to the computation in the proof of { Proposition \ref{prop-compare-constant}}, we can obtain
\begin{equation*}
\label{local flat 2 ge}
    \begin{split}
    Y_\alpha(M^n,g)\leq &J_{g,\alpha}(u)\\
    \leq & Y_\alpha(\mathbb{S}^n, g_0)+C\left(\frac\delta\lambda\right)^{-n}\\ &-A_0 \cdot \frac{\int_{B_\delta\times B_\delta}|x-y|_g^{\alpha-2}f_\lambda(x)f_\lambda(y)dxdy} {\|u\|_{L^{2n/(n+\alpha)}(M^n)}^2}.
    \end{split}
\end{equation*}
Since
\begin{equation*}\label{local flat 3 ge}
    \begin{split}
    &\int_{B_\delta\times B_\delta}|x-y|_g^{\alpha-2}f_\lambda(x)f_\lambda(y)dxdy\\ =&\lambda^{-(n+\alpha)}\int_{B_\delta\times B_\delta}|x-y|_g^{\alpha-2} \left(1+\frac{|x|^2}{\lambda^2}\right)^{-\frac{n+\alpha}2} \left(1+\frac{|y|^2}{\lambda^2}\right)^{-\frac{n+\alpha}2}dxdy\\
    =&\lambda^{n-2}\int_{B_{\delta/\lambda}\times B_{\delta/\lambda}}|u-v|^{\alpha-2} (1+|u|^2)^{-\frac{n+\alpha}2} (1+|v|^2)^{-\frac{n+\alpha}2} dudv\\
    \geq&C_1\lambda^{n-2},
    \end{split}
\end{equation*}
then
\begin{equation*}\label{local flat 4 ge}
    \begin{split}
    &-C\left(\frac\delta\lambda\right)^{-n}+A_0\frac{\int_{B_\delta\times B_\delta}|x-y|_g^{\alpha-2}f_\lambda(x)f_\lambda(y)dxdy} {\|u\|_{L^{2n/(n+\alpha)}(M^n)}^2}\\
    \geq&-C\left(\frac\delta\lambda\right)^{-n}+C_2A_0\lambda^{n-2} =\lambda^{n-2}(C_2A_0-C\lambda^2\delta^{-n})>0
    \end{split}
\end{equation*}
by choosing $\lambda$ much smaller than $\delta$. Therefore, we deduce that
$$
    Y_\alpha(M^n,g)<Y_\alpha(\mathbb{S}^n,g_0).
$$
\end{proof}

\begin{rem}
Due to the lack of local sharp  inequality for the  case of $\alpha>n$, it is not clear what is the form for the Aubin type $\epsilon-$inequality. It is also very interesting to analysz the blowup behavior of solutions to the equations with negative power, since  the concentration compactness principle does not hold, and the classical Nash-Moser type iteration does not work neither due to the lact of local sharp inequality.
\end{rem}

\begin{rem}
While we are working on this paper,  M. Zhu was informed by F. Hang and P. Yang of their recent work on $Q-$ curvature problem on 3 manifolds  \cite{HY2014-1}, where their estimates relies on the crucial sharp Sobolev inequality originally proved by Yang and Zhu \cite{YZ2004}. It seems that their argument is hard to be extended for operator with fractional order. The recent discovery of the reversed sharp  Hardy-Littlewood-Sobolev inequality \cite{DZ2} is the foundation for our current work for the case of $\alpha>n$. A unified approach for the Nirenberg problem for $\alpha<n$ was given in a recent paper \cite{JLX2014c}.
\end{rem}






\noindent{\bf ACKNOWLEDGMENT.} M. Zhu would like to thank F. Hang and  P. Yang for some conversations that shape this paper. He also thank them for bringing \cite{HY2014-1}-\cite{HY2015} to his attention; and thank T. Jin  for bringing reference \cite{DPV2011}  to his attention.  M. Zhu is partially supported by a collaboration grant from Simons Foundation. The work of Y. Han is partially supported by the National Natural Science Foundation of China (Grant No. 11201443 and 11101319).

\end{document}